\documentclass[a4paper,10pt,leqno]{amsart}
        \title{$K$-theory and actions on Euclidean retracts}
       \author{Bartels, A.}
       \address{WWU M\"unster\\
               Mathematisches Institut\\
               Einsteinstr.~62,
               D-48149 M\"unster, Germany}
        \email{a.bartels@wwu.de}
        \urladdr{http://www.math.uni-muenster.de/u/bartelsa} 
         \date{December 2017}

  \usepackage{comment}
  \usepackage{hyperref}
  \usepackage{calc}
  \usepackage{enumerate,amssymb}
  \usepackage[arrow,curve,matrix,tips,2cell]{xy}
    \SelectTips{eu}{10} \UseTips
    \UseAllTwocells
  \usepackage{tikz}
  \usepackage{pdfcolmk}
  
  \DeclareMathAlphabet{\matheurm}{U}{eur}{m}{n}

  \newcommand{\IN}{\mathbb{N}}

  \newcommand{\IQ}{\mathbb{Q}}
  \newcommand{\IR}{\mathbb{R}}

  \newcommand{\IZ}{\mathbb{Z}}

  \newcommand{\cala}{\mathcal{A}}
  
  \newcommand{\calc}{\mathcal{C}}

  \newcommand{\calf}{\mathcal{F}}

  \newcommand{\calm}{\mathcal{M}}

  \newcommand{\calp}{\mathcal{P}}

  \newcommand{\calt}{\mathcal{T}}
  \newcommand{\calu}{\mathcal{U}}
  \newcommand{\calv}{\mathcal{V}}
  \newcommand{\calw}{\mathcal{W}}

  \newcommand{\bfK}{{\mathbf K}}

  \newcommand{\bfY}{{\mathbf Y}}

  \newcounter{commentcounter}

  \newcommand{\inred}[1]{#1}

  \newcommand{\ignore}[1]{}

  \theoremstyle{plain}
  \newtheorem{theorem}{Theorem}[section]
  \newtheorem{lemma}[theorem]{Lemma}
  
  \newtheorem{proposition}[theorem]{Proposition}
  \newtheorem{conjecture}[theorem]{Conjecture}

  \theoremstyle{definition}
  \newtheorem{definition}[theorem]{Definition}
  \newtheorem*{definition*}{Definition}

  \theoremstyle{remark}
  \newtheorem{remark}[theorem]{Remark}  
    
  \newtheorem{example}[theorem]{Example}

  \makeatletter\let\c@equation=\c@theorem\makeatother

   {\end{list}}

  \DeclareMathOperator{\Fin}{Fin}

  \DeclareMathOperator{\Mod}{Mod}

  \DeclareMathOperator{\Prob}{Prob}

  \DeclareMathOperator{\GL}{GL}
  
  \DeclareMathOperator{\SL}{SL}
  \DeclareMathOperator{\Wh}{Wh}
  \DeclareMathOperator{\VCyc}{VCyc}

  \newcommand{\CAT}{\operatorname{CAT}}
  \newcommand{\CF}{{\mathit{CF}}}
  
  \newcommand{\dd}{{\partial}}
  \newcommand{\e}{{\varepsilon}}
   
  \newcommand{\fol}{{\mathit{fol}}}
  \newcommand{\FS}{{\mathit{FS}}}

  \newcommand{\ox}{{\otimes}}
  \newcommand{\PMF}{{\mathit{\calp \hspace{-.6ex} \calm \hspace{-.1ex} \calf}}}

  \newcommand{\x}{{\times}}
  
\begin{document}

  \maketitle

  \begin{abstract}
    This note surveys axiomatic results for the Farrell-Jones Conjecture in terms of actions on Euclidean retracts and applications of these to $\GL_n(\IZ)$, relative hyperbolic groups and mapping class groups.
  \end{abstract}
  

  \section*{Introduction}
  
  Motivated by surgery theory Hsiang~\cite{Hsiang-ICM83} made a number of influential conjectures about the $K$-theory of integral group rings $\IZ[G]$ for torsion free groups $G$.
  These conjectures often have direct implications for the classification theory of manifolds of dimension $\geq 5$.
  A good example is the following.
  An $h$-cobordism is a compact manifold $W$ that has two boundary components $M_0$ and $M_1$ such that both inclusions $M_i \to W$ are homotopy equivalences.
  The Whitehead group $\Wh(G)$ is the quotient of $K_1(\IZ[G])$ by the subgroup generated by the canonical units $\pm g$, $g \in G$. 
  Associated to an $h$-cobordism is an invariant, the Whitehead torsion, in $\Wh(G)$, where $G$ is the fundamental group of $W$.  
  A consequence of the $s$-cobordism theorem is that for $\dim W \geq 6$, an $h$-cobordism $W$ is trivial (i.e., isomorphic to a product $M_0 \x [0,1]$) iff its Whitehead torsion vanishes.
  Hsiang conjectured that for $G$ torsion free $\Wh(G) = 0$, and thus that in many cases $h$-cobordisms are products.
  
  The Borel conjecture asserts that closed aspherical manifolds are topologically rigid, i.e., that any homotopy equivalence to another closed manifold is homotopic to a homeomorphism.
  The last step in proofs of instances of this conjecture via surgery theory uses a vanishing result for $\Wh(G)$ to conclude that an $h$-cobordism is a product and that therefore the two boundary components are homeomorphic. 
   
  Farrell-Jones~\cite{Farrell-Jones-Ktheory-dynamics-I} pioneered a method of using the geodesic flow on non-positively curved manifolds to study these conjectures.
  This created a beautiful connection between $K$-theory and dynamics that led Farrell-Jones~\cite{Farrell-Jones-TopRigigdityCompactNonPos}, among many other results, to a proof of the Borel Conjecture for   closed Riemannian manifolds of non-positive curvature of dimension $\geq 5$.   
  Moreover, Farrell-Jones~\cite{Farrell-Jones-IsomorphismConjectures} formulated (and proved many instances of) a conjecture about the structure of the algebraic $K$-theory (and $L$-theory) of group rings, even in the presence of torsion in the group.
  Roughly, the Farrell-Jones Conjecture states that the main building blocks for the $K$-theory of $\IZ[G]$ is the $K$-theory of $\IZ[V]$ where $V$ varies of the family of virtually cyclic subgroups of $G$.
  It implies a number of other conjectures, among them Hsiang's conjectures, the Borel Conjecture in  dimension $\geq 5$, the Novikov Conjecture on the homotopy invariant of higher signatures, Kaplansky's conjecture about idempotents in group rings, see~\cite{Lueck-ICM2010} for a summary of these and other applications. 
  
  My goal in this note is twofold. 
  The first goal is to explain a condition formulated in terms of existence of certain actions of $G$ on Euclidean retracts that implies the Farrell-Jones Conjecture for $G$.
  This condition was developed in joint work with L\"uck and Reich~\cite{Bartels-Lueck-Borel, Bartels-Lueck-Reich-FJ-hyperbolic} where the connection between $K$-theory and dynamics has been extended beyond the  context of Riemannian manifolds to prove the Farrell-Jones Conjecture for hyperbolic and $\CAT(0)$-groups.  
  The second goal is to outline how this condition has been used in joint work with L\"uck, Reich and R\"uping and with Bestvina to prove the Farrell-Jones Conjecture for $\GL_n(\IZ)$ and mapping class groups.
  A common difficulty for both families of groups is that their natural proper actions (on the associated symmetric space, respectively on Teichm\"uller space) is not cocompact.
  In both cases the solution depends on a good understanding of the action away from cocompact subsets and an induction on a complexity of the groups.
  As a preparation for mapping class groups we also discuss relatively hyperbolic groups.
  
  The Farrell-Jones Conjecture has a prominent relative, the Baum-Connes Conjecture for topological $K$-theory of group $C^*$-algebras~\cite{Baum-Connes-Geometric-K-foliations, Baum-Connes-Higson-class-spaces-proper-act-K}.
  The two conjectures are formally very similar, but methods of proofs are different. 
  In particular, the conditions discussed in Section~\ref{sec:actions} are not known to imply the Baum-Conjecture.
  The classes of groups for which the two conjectures are known differ. 
  For example, by work of Kammeyer-L\"uck-R\"uping~\cite{Kammeyer-Lueck-Rueping-FJ-lattices} all lattice in Lie groups satisfy the Farrell-Jones Conjecture; despite Lafforgues~\cite{Lafforgue-K-bivariante-Banach-BC} positive results for many property $T$ groups, the Baum-Connes Conjecture is still a challenge for $\SL_3(\IZ)$.
  Wegner~\cite{Wegner-FJsolvable} proved the Farrell-Jones Conjecture for all solvable groups, but the case of amenable (or just elementary amenable) groups is open;
  in contrast Higson-Kasparov~\cite{Higson-Kasparov-E-KK-properly-on-Hilbert} proved the Baum-Connes Conjecture for all a-T-menable groups, a class of groups that contains all amenable groups. 
  On the other hand, hyperbolic groups satisfy both conjectures.
  See Mineyev-Yu~\cite{Mineyev-Yu-BC-hyperbolic} and Lafforgue~\cite{Lafforgue-BC-hyperbolic} for the Baum-Connes Conjecture and, as mentioned above, \cite{Bartels-Lueck-Borel, Bartels-Lueck-Reich-FJ-hyperbolic} for the Farrell-Jones Conjecture. 
  For a more comprehensive summary of the current status of the Farrell-Jones Conjecture the reader is directed to~\cite{Lueck-MFO-slides-FJC, Reich-Varisco-FJ-Survey}.

  \subsection*{Acknowledgement}
  It is a pleasure to thank my teachers, coauthors, and students for the many things they taught me.
  The work described here has been supported by the SFB 878 in M\"unster.

  \section{The Formulation of the Farrell-Jones Conjecture} \label{sec:formulation-FJC}


  
  \subsection*{Classifying spaces for families}  
  A family $\calf$ of subgroups of a group $G$ is a non-empty collection of subgroups that is closed under conjugation and subgroups.
  Examples are the family $\Fin$ of finite subgroups and the family $\VCyc$ of virtually cyclic subgroups (i.e., of subgroups containing a cyclic subgroup as a subgroup of finite index).
  For any family of subgroups of $G$ there exists a $G$-$CW$-complex $E_\calf G$ with the following property: if $E$ is any other $G$-$CW$-complex such that all isotropy groups of $E$ belong to $\calf$, then there is a up to $G$-homotopy unique $G$-map $E \to E_\calf G$.
  This space is not unique, but it is unique up to $G$-homotopy equivalence.
  Informally one may think about $E_\calf G$ as a space that encodes the group $G$ relative to all subgroups from $\calf$.
  Often there are interesting geometric models for this space, in particular for $\calf = \Fin$.
  More information about this space can be found for example in~\cite{Lueck-SurveyClassifyingSpacesFamilies}.
  An easy way to construct $E_\calf G$ is as the infinite join $\ast_{i=0}^\infty (\coprod_{F \in \calf} G/F)$.
  If $\calf$ is closed under supergroups of finite index (i.e., if $F \in \calf$ is a subgroup of finite index in $F'$, then also $F' \in \calf$), then the full simplicial complex on $\coprod_{F \in \calf} G/F$ is also a model for $E_\calf G$; we will denote this model later by $\Delta_\calf(G)$.
    

  \subsection*{The formulation of the conjecture}
  The original formulation of the Farrell-Jones Conjecture~\cite{Farrell-Jones-IsomorphismConjectures} used homology with coefficients in stratified and twisted $\Omega$-spectra.
  Here we use the equivalent~\cite{Hambleton-Pedersen-Identifying} formulation developed by Davis-L\"uck~\cite{Davis-Lueck-assembly-maps}. 
  Given a ring $R$, Davis-L\"uck construct a homology theory $X \mapsto H_*^G(X;\bfK_{R})$ for $G$-spaces
  with the property that $H^G_*(G/H;\bfK_R) \cong K_*(R[H])$.
  
  Let $\calf$ be a family of subgroups of the group $G$.
  Consider the projection map $E_\calf G \to G/G$ to the one-point $G$-space $G/G$.
  It induces the $\calf$-assembly map 
  \begin{equation*}
  	\alpha^G_\calf \colon H^G_*(E_\calf;\bfK_R) \to H^G_*(G/G;\bfK_R) \cong K_* (R[G]).
  \end{equation*}
  
  \begin{conjecture}[Farrell-Jones Conjecture] \label{conj:FJC}
  	 For any group $G$ and any ring $R$ the assembly map $\alpha^G_{\VCyc}$ is an isomorphism.
  \end{conjecture}
  
  This version of the conjecture has been stated in~\cite{Bartels-Farrell-Jones-Reich-OnIso}.
  The original formulation of Farrell and Jones~\cite{Farrell-Jones-IsomorphismConjectures} considered only the integral group ring $\IZ[G]$.  
  Moreover, Farrell and Jones wrote that they regard this and related conjectures \emph{only as estimates which best fit the know data at this time}.
  However, the conjecture is still open and does still fit with all known data today. 
  
  \subsection*{Transitivity principle}
  Informally one can view the statement that the assembly map $\alpha^G_\calf$ is an isomorphism for a group $G$ and a ring $R$ as the statement that $K_*(R[G])$ can be assembled from $K_*(R[F])$ for all $F \in \calf$ (and group homology).
  If $\calv$ is  a family of subgroups of $G$ that contains all subgroups from $\calf$, then one can apply this slogan in two steps, for $G$ relative to $\calv$ and for each $V \in \calv$ relative to the $F \in \calf$ with $F \subseteq V$.
  The implementation of this is the following transitivity principle.
      
  \begin{theorem}[\cite{Farrell-Jones-IsomorphismConjectures, Lueck-Reich-Survey-BC-FJ}]
     \label{thm:transitivity-principle}
  	For $V \in \calv$ set $\calf_V := \{ F \mid F \in \calf, F \subseteq V \}$.
  	Assume that for all $V \in \calv$ the assembly map $\alpha^V_{\calf_V}$ is an isomorphism.
  	Then $\alpha^G_\calf$ is an isomorphism iff $\alpha^G_\calv$ is an isomorphism.
  \end{theorem}
       
       
  \subsection*{Twisted coefficients} 
  Often it is beneficial to study more flexible generalizations of Conjecture~\ref{conj:FJC}.
  Such a generalization is the Fibred Isomorphism Conjecture of Farrell-Jones~\cite{Farrell-Jones-IsomorphismConjectures}.
  An alternative is the Farrell-Jones Conjecture with coefficients in additive categories~\cite{Bartels-Reich-Coeff}, here one allows additive categories with an action of a group $G$ instead of just a ring as coefficients.
  This version of the conjecture applies in particular to twisted group rings.
  These generalizations of the Conjecture have better inheritance properties.
  Two of these inheritance property are stability under directed colimits of groups, and stability under taking subgroups.
  For a summary of the inheritance properties see~\cite[Thm.~27(2)]{Reich-Varisco-FJ-Survey}.
  Often proofs of cases of the Farrell-Jones Conjecture use these inheritance properties in inductions or to reduce to special cases. 
  We will mean by the statement that $G$ satisfies the Farrell-Jones Conjecture relative to $\calf$, that the assembly map $\alpha^G_\calf$ is bijective for all additive categories $\cala$ with $G$-action.
  However, this as a technical point that can be safely ignored for the purpose of this note.        
       
  \subsection*{Other theories}
  The Farrell-Jones Conjecture for $K$-theory discussed so far has an analog in $L$-theory, as it appears in surgery theory.
  For some of the applications mentioned before this is crucial.
  For example, the Borel Conjecture for a closed aspherical manifold $M$ of dimensions $\geq 5$ holds if the fundamental group of $M$ satisfies both the $K$-and $L$-theoretic Farrell-Jones Conjecture.
  However, proofs of the Farrell-Jones Conjecture in $K$- and $L$-theory are by now very parallel.
  Recently, the techniques for the Farrell-Jones Conjecture in $K$- and $L$-theory have been extended to also cover Waldhausen's $A$-theory~\cite{Enkelmann-Lueck-Pieper-Ullmann-Winges-OnA-Theory-FJ, Kasprowski-Ullmann-Wegner-Winges-A-FJ-solv, Ullmann-Winges-A-theory-Farrell-Hsiang}.
  In particular, the conditions we will discuss in Section~\ref{sec:actions} are now known to imply the  Farrell-Jones Conjecture in all three theories.

  \section{Actions on compact spaces} \label{sec:actions}
  
  \subsection*{Amenable actions and exact groups}

  \begin{definition}[Almost invariant maps]
    Let $X$, $E$ be $G$-spaces where $E$ is equipped with a $G$-invariant metric $d$.
  	We will say that a sequence of maps $f_n \colon X \to E$ is almost $G$-equivariant if for any $g \in G$
  	\begin{equation*}
  		\sup_{x \in X} d(f_n(gx), gf_n(x)) \to 0 \quad \text{as} \quad n \to \infty.
  	\end{equation*} 
  \end{definition}

  For a discrete group $G$ we equip the space $\Prob(G)$ of probability measures on $G$ with the metric it inherits as subspace of $l^1(G)$.
  This metric generates the topology of point-wise convergence on $\Prob(G)$. 
  We recall the following definition. 
  
  \begin{definition}
  	An action of a group $G$ on a compact space $X$ is said to be amenable if there exists a sequence of almost equivariant maps $X \to \Prob(G)$.
  \end{definition}
    
  A group is amenable iff its action on the one point space is amenable.
  Groups that admit an amenable action on a compact Hausdorff space are said to be \emph{exact} or \emph{boundary amenable}.
  The class of exact groups contains all amenable groups, hyperbolic groups~\cite{Adams-BoundaryAmenablityHyperbolicGr}, and all linear groups~\cite{Guentner-Higson-Weinberger-NovikovLinearGroup}.
  Other prominent groups that are known to be exact are mapping class groups~\cite{Hamenstaedt-geometry-MCG-I, Kida-MCG-measure-equivalence} and the group of outer automorphisms of free groups~\cite{Bestvina-Guirardel-Horbez-d-amenabilty-OutF_n}.
  The Baum-Connes assembly map is split injective for all exact groups~\cite{Higson-bivariant-Novikov-GAFA, Yu-CBCforUniformEmbedding}.
  This implies the Novikov conjecture for exact groups. 
  This is an analytic result for the Novikov conjecture, in the sense that it has no known proof that avoids the Baum-Connes Conjecture.
  There is no corresponding injectivity result for assembly maps in algebraic $K$-theory.   
  For a survey about amenable actions and exact groups see~\cite{Ozawa-ICM}.

  \subsection*{Finite asymptotic dimension}
  Results for assembly maps in algebraic $K$-theory and $L$-theory often depend on a finite dimensional setting; the space of probability measures has to be replaced with a finite dimensional space.
  We write $\Delta(G)$ for the full simplicial complex with vertex set $G$ and $\Delta^{(N)}(G)$ for its $N$-skeleton.
  The space $\Delta(G)$ can be viewed as the space of probability measures on $G$ with finite support. 
  We equip $\Delta(G)$ with the $l^1$-metric; this is the metric it inherits from $\Prob(G)$.
  
  \begin{definition}[$N$-amenable action]
  	We will say that an action of a group $G$ on a compact space $X$ is $N$-amenable if there exists a sequence of almost equivariant maps $X \to \Delta^{(N)}(G)$.
  \end{definition}   
  
  The natural action of a countable group $G$ on its Stone-\v{C}ech compactification $\beta G$ is $N$-amenable iff the asymptotic dimension of $G$ is at most $N$~\cite[Thm.~6.5]{Guentner-Willett-Yu-dyn-asy-dim}. 
  This condition (for any $N$) also implies exactness and therefore the Novikov conjecture~\cite{Higson-Roe-AmenableActionsNovikov}.
  For groups $G$ of finite asymptotic dimension for which in addition the classifying space $BG$ can be realized as a finite $CW$-complex, there is an alternative argument for the Novikov Conjecture~\cite{Yu-Novikov-finite-asym-dim} that has been translated to integral injectivity results for assembly maps in algebraic $K$-theory and $L$-theory~\cite{Bartels-Squeezing, Carlsson-Goldfarb-IntegralNovikovAsyDim}. 
  These injectivity results have seen far reaching generalizations to groups of finite decomposition complexity~\cite{Guentner-Tessera-Yu-geom-complx-top-rigid-FDC, Kasprowski-On-FDC-K-theory, Ramras-Tessera-Yu-FDC-algKtheory}.

  \subsection*{$N$-$\calf$-amenable actions.}
  Constructions of transfer maps in algebraic $K$-theory and $L$-theory often depend on actions on spaces that are much nicer than $\beta G$.
  A good class of spaces to use for the Farrell-Jones Conjecture are Euclidean retracts, i.e., compact spaces that can be embedded as a retract in some $\IR^n$.
  Brouwer's fixed point theorem implies that for an action of a group on an Euclidean retract any cyclic subgroup will have a fixed point.
  It is not difficult to check that this obstructs the existence of almost equivariant maps to $\Delta^{(N)}$ (assuming $G$ contains an element of infinite order).
  Let $\calf$ be a family of subgroups of $G$ that is closed under taking supergroups of finite index.
  Let $S := \coprod_{F \in \calf} G/F$ be the set of all left cosets to members of $\calf$.
  Let $\Delta_\calf(G)$ be the full simplicial complex on $S$ and $\Delta_\calf^{(N)}(G)$ be its $N$-skeleton.
  We equip $\Delta_\calf(G)$ with the $l^1$-metric.
  
  \begin{definition}[$N$-$\calf$-amenable action]
  	We will say that an action of $G$ on a compact space $X$ is $N$-$\calf$-amenable if there exists a sequence of almost equivariant maps $X \to \Delta_\calf^{(N)}(G)$.
  	If an action is $N$-$\calf$-amenable for some $N \in \IN$, then we say that it is finitely $\calf$-amenable.
  \end{definition}  
  
  \begin{remark} \label{rem:N-F-amenable-G-CW}
    Let $X$ be a $G$-$CW$-complex with isotropy groups in $\calf$ and of dimension $\leq N$.
    As $\Delta_\calf(G)$ is a model for $E_\calf G$ we obtain a cellular $G$-map $f \colon X \to \Delta^{(N)}_\calf(G)$; this map is also continuous for the $l^1$-metric.
    In particular, the constant sequence $f_n \equiv f$ is almost equivariant.
    Therefore one can view $N$-$\calf$-amenability for $G$-spaces as a relaxation of the property of being a $G$-$CW$-complex with isotropy in $\calf$ and of dimension $\leq N$.
        
    This relaxation is necessary to obtain compact examples and reasonably small $\calf$: If a $G$-$CW$-complex is compact,  then it has only finitely many cells.
    In particular, for each cell the isotropy group has finite index in $G$, so $\calf$ would have to contain subgroups of finite index in $G$.   
  \end{remark}

  
  
  \begin{theorem}  [\cite{Bartels-Lueck-Borel, Bartels-Lueck-Reich-FJ-hyperbolic}] \label{thm:FJ-N-F-actions}
    Suppose that $G$ admits a finitely $\calf$-amenable action on a Euclidean retract.
  	Then $G$ satisfies the Farrell-Jones Conjecture relative to $\calf$.
  \end{theorem}
  
  \begin{remark}
  	The proof of Theorem~\ref{thm:FJ-N-F-actions} depends on methods from controlled topology/algebra that have a long history. 
  	An introduction to controlled algebra is given in~\cite{Pedersen-contr-alg-K-survey};
  	an introduction to the proof of Theorem~\ref{thm:FJ-N-F-actions} can be found in~\cite{Bartels-On-proofs}.
  	Here we only sketch a very special case, where these methods are not needed. 
  	
  	Assume that the Euclidean retract is a $G$-$CW$-complex $X$.
  	As pointed out in Remark~\ref{rem:N-F-amenable-G-CW} this forces $\calf$ to contain subgroups of finite index in $G$.
  	As $X$ is contractible, the cellular chain complex of $X$ provides a finite resolution $C_*$ over $\IZ[G]$ of the trivial $G$-module $\IZ$.
  	Note that in each degree $C_k = \bigoplus \IZ[G/F_{i}]$ is a finite sum of permutation modules with $F_{i} \in \calf$ and $F_i$ of finite index in $G$.  
  	For a finitely generated projective $R[G]$-module $P$ we obtain a finite resolution $C_* \ox_\IZ P$ of $P$.
    Each module in the resolution is a finite sum of modules of the form $\IZ[G/F] \ox_\IZ P$ with $F \in \calf$ and of finite index in $G$.
    Here $\IZ[G/F] \ox_\IZ P$ is equipped with the diagonal $G$-action and can be identified with the $R[G]$-module obtained by first restricting $P$ to an $R[F]$-module and then inducing back up from $R[F]$ to $R[G]$.
  	In particular, $\big[\IZ[G/F] \ox_\IZ P\big] \in K_0(R[G])$ is in the image of the assembly map relative to the family $\calf$.
  	It follows that $[P] = \sum_{k} (-1)^k [C_k \ox_\IZ P]$ is also in the image.
  	Therefore the assembly map $H_0(E_\calf G;\bfK_R) \to K_0(R[G])$ is surjective.  
  	(This argument did not use that $\calf$ is closed under supergroups of finite index.)
  \end{remark}
  
  \begin{example}
  	Let $G$ be a hyperbolic group. 
  	Its Rips complex can be compactified to a Euclidean retract~\cite{Bestvina-Mess-boundary}.
  	The natural action of $G$ on this compactification is finitely $\VCyc$-amenable~\cite{Bartels-Lueck-Reich-Cover}.
  	
  	To obtain further examples of finitely $\calf$-amenable actions on Euclidean retracts, it is helpful to replace $\VCyc$ with a larger family of subgroups $\calf$.
  	Groups that act acylindrically hyperbolic on a tree admit finitely $\calf$-amenable actions on Euclidean retracts where $\calf$ is the family of subgroups that is generated by the virtually cyclic subgroups and the isotropy groups for the original action on the tree~\cite{Knopf-AcylTree-FJC}.
  	Relative hyperbolic groups and mapping class groups are discussed in Section~\ref{sec:cover-at-infty}. 
  \end{example}
  
  	
  	

  \begin{remark}
    A natural question is which groups admit finitely $\VCyc$-amenable actions on Euclidean retracts.
    A necessary condition for an action to be finitely $\VCyc$-amenable is that all isotropy groups of the action are virtually cyclic.
    Therefore, a related question is which groups admit actions on Euclidean retracts such that all isotropy groups are virtually cyclic.
    The only groups admitting such actions that I am aware of are hyperbolic groups.
    In fact, I do not even know whether or not the group $\IZ^2$ admits an action on a Euclidean retract (or on a disk) such that all isotropy groups are virtually cyclic.
    There are actions of $\IZ^2$ on disks without a global fixed point.
    This is a consequence of Oliver's analysis of actions of finite groups on disks~\cite{Oliver-fixed-points-finite-acyclic}.
    On the other hand, there are finitely generated groups for which all actions on Euclidean retracts have a global fixed point~\cite{Arzhantseva-Bridson-Januszkiewicz-Leary-Minasyan-Switkowski}. 
  \end{remark}
  
  	
  	
   
  \subsection*{Homotopy actions} 
    
  There is a generalization of Theorem~\ref{thm:FJ-N-F-actions} using homotopy actions.
  In order to be applicable to higher $K$-theory these actions need to be homotopy coherent.
  The passage from strict actions to homotopy actions is already visible in the work of Farrell-Jones where it corresponds to the passage from the asymptotic transfer used for negatively curved manifolds~\cite{Farrell-Jones-Ktheory-dynamics-I} to the focal transfer used for non-positively curved manifolds~\cite{Farrell-Jones-TopRigigdityCompactNonPos}. 
  
  \begin{definition}[\cite{Vogt-HomotopyLimitsColimits,Wegner-CAT0}]
  	A \emph{homotopy coherent action} of a group $G$ on a space $X$ is a continuous map
  	\begin{equation*}
  		\Gamma \colon \coprod_{j=0}^\infty ((G \x [0,1])^j \x G \x X) \to X 
  	\end{equation*}
  	such that
  	\begin{equation*}
  		\Gamma(g_k,t_k,\dots,t_1,g_0,x) = 
  		   \begin{cases} \Gamma(g_k,\dots,g_j,\Gamma(g_{j-1},\dots,x)) & t_j = 0 \\
  		                 \Gamma(g_k,\dots,g_j g_{j-1},\dots,x) & t_j = 1 \\
  		                 \Gamma(g_k,\dots,t_2,g_1,x) & g_0 = e, 0 < k \\
  		                 \Gamma(g_k,\dots,t_{j+1}t_j,\dots,g_0,x) & g_j = e, 1 \leq j < k \\
  		                 \Gamma(g_{k-1},\dots,t_1,g_0,x) & g_k = e, 0 < k \\
  		                 x & g_0 = e, k=0  	
  		   \end{cases}
  	\end{equation*}
  \end{definition}
  
  Here $\Gamma(g,-) \colon X \to X$ should be thought of the action of $g$ on $X$, the map $\Gamma(g,-,h,-) \colon [0,1] \x X \to X$ is a homotopy from $\Gamma(g,-) \circ \Gamma(h,-)$ to $\Gamma(gh,-)$ and the remaining data in $\Gamma$ encodes higher coherences. 
  
  In order to obtain sequences of almost equivariant maps for homotopy actions it is useful to also allow the homotopy action to vary.  
  
  \begin{definition}[$N$-$\calf$-amenability for homotopy coherent actions]
  	A sequence of homotopy coherent actions $(\Gamma_n,X_n)$ of a group $G$ is said to be $N$-$\calf$-amenable if there exists a sequence of continuous maps $f_n \colon X_n \to \Delta^{(N)}_\calf(G)$ such that for all $k$ and all $g_k,\dots,g_0 \in G$ 
  	\begin{equation*}
  		\sup_{x \in X, t_k,\dots,t_1 \in [0,1]}  d(f_n(\Gamma(g_k,t_k,\dots,t_1,g_0,x),g_k\cdots g_0 f_n(x))  \to 0 \quad \text{as} \quad n \to \infty.
  	\end{equation*}     	
  \end{definition}
  
  \begin{theorem}[\cite{Bartels-Lueck-Borel, Wegner-CAT0}] \label{thm:FJ-homotopy-coherent}
    Suppose that $G$ admits a sequence of homotopy coherent actions on Euclidean retracts of uniformly bounded dimension that is finitely $\calf$-amenable.
  	Then $G$ satisfies the Farrell-Jones Conjecture relative to $\calf$.	
  \end{theorem}
  
  \begin{remark} \label{rem:covers-not-maps}
  	Groups satisfying the assumptions of Theorem~\ref{thm:FJ-homotopy-coherent} are said to be homotopy transfer reducible in~\cite{Enkelmann-Lueck-Pieper-Ullmann-Winges-OnA-Theory-FJ}.
  	The original formulations of Theorems~\ref{thm:FJ-N-F-actions} and~\ref{thm:FJ-homotopy-coherent} were not in terms of almost equivariant maps, but in terms of certain open covers of $G \x X$.
  	
  	We recall here the formulation used for actions. 
  	(The formulation for homotopy actions is more cumbersome.)
  	A subset $U$ of a $G$-space is said to be an $\calf$-subset if there is $F \in \calf$ such that $gU = U$ for all $g \in F$ and $U \cap gU = \emptyset$ for all $g \in G \setminus F$.
  	A collection $\calu$ of subsets is said to be $G$-invariant if $gU \in \calu$ for all $g \in G$, $U \in \calu$.
  	If no point is contained in more than $N+1$ members of $\calu$, then $\calu$ is said to be of order (or dimension) $\leq N$.
  	A $G$-invariant cover by open $\calf$-subsets is said to be an $\calf$-cover.

  	For a compact $G$-space $X$ we equip now $G \x X$ with the diagonal $G$-action.
  	For $S \subseteq G$ finite an $\calf$-cover of $G \x X$ is said to be $S$-wide (in the $G$-direction) if
  	\begin{equation*}
  		\forall \, (g,x) \in G \x X \; \; \exists \, U \in \calu \;\; \text{such that} \;\; gS \x \{ x \} \subseteq U.
  	\end{equation*}
  	Then the action of $G$ on $X$ is $N$-$\calf$-amenable iff for any $S \subset G$ finite there exists an $S$-wide $\calf$-cover $\calu$ for $G \x X$ of dimension at most $N$~\cite[Prop.~4.5]{Guentner-Willett-Yu-dyn-asy-dim}. 
  	A translation from covers to maps is also used in~\cite{Bartels-Lueck-Borel, Bartels-Lueck-Reich-FJ-hyperbolic, Wegner-CAT0}.  
  	From the point of view of covers (and because of the connection to the asymptotic dimension) it is natural to think of the $N$ in $N$-$\calf$-amenable as a kind of dimension for the action of $G$ on $X$, see~\cite{Guentner-Willett-Yu-dyn-asy-dim, Sawicki-eq-asym-dim}.  
  	  	 
  	A further difference between the formulations used above and in the references given is that the conditions on the topology of $X$ are formulated differently, but certainly Euclidean retracts satisfy the condition from~\cite{Bartels-Lueck-Borel}.
  \end{remark}

  \begin{example}
  	Theorem~\ref{thm:FJ-homotopy-coherent} applies to $\CAT(0)$-groups where $\calf=\VCyc$ is the family of virtually cyclic subgroups~\cite{Bartels-Lueck-CAT0-geod-flow, Wegner-CAT0}.
  	An application of Theorem~\ref{thm:FJ-homotopy-coherent} to $\GL_n(\IZ)$ will be discussed in Section~\ref{sec:cover-at-infty}.
  \end{example}
  	
  \begin{remark}[The Farrell-Hsiang method]
    There are interesting groups for which one can deduce the Farrell-Jones Conjecture using Theorems~\ref{thm:FJ-N-F-actions} (or~\ref{thm:FJ-homotopy-coherent}) and inheritance properties. 
    However, it is not clear that these methods can account for all that is currently known.
    A third method, going back to work of Farrell-Hsiang~\cite{Farrell-Hsiang-top-Eucl-space-form}, combines induction results for finite groups~\cite{Dress-Induction-structure-orth-reps, Swan-Induced-reps-proj-modules} with controlled topology/algebra.
    An axiomatization of this method is given in~\cite{Bartels-Lueck-Farrell-Hsiang}.
  	In important part of the proof of the Farrell-Jones Conjecture for solvable groups~\cite{Wegner-FJsolvable} is a combination of this method with Theorem~\ref{thm:FJ-homotopy-coherent}. 
  \end{remark}	
  	
  \begin{remark}[Trace methods]
  	The $K$-theory Novikov Conjecture concerns injectivity of assembly maps in algebraic $K$-theory, i.e., lower bounds for the algebraic $K$-theory of group rings.
  	For the integral group ring of groups, that are only required to satisfy a mild homological finiteness assumption, trace methods have been used by B\"okstedt-Hsiang-Madsen~\cite{Boekstedt-Hsiang-Madsen-cyclotomic} and L\"uck-Reich-Rognes-Varisco~\cite{Lueck-Reich-Rognes-Varisco-K-theory-cyclotomic-trace} to obtain rational injectivity results.
  	The latter result in particular yields interesting lower bounds for Whitehead groups.
  	For the group ring over the ring of Schatten class operators Yu~\cite{Yu-Novikov-Schatten-class} proved rational injectivity of the Farrell-Jones assembly map for all groups.
  	This is the only result I am aware of for the Farrell-Jones Conjecture that applies to all groups!
  \end{remark}

  \section{Flow spaces} \label{sec:flow-spaces}
  
  The construction of almost equivariant maps often uses the dynamic of a flow associated to the situation.
  
  \begin{definition}
    A \emph{flow space} for a group $G$ is a metric space $\FS$ equipped with a flow $\Phi$ and an isometric $G$-action where the flow and the $G$-action commute. 
    For $\alpha > 0$, $\delta > 0$, $c,c' \in \FS$ we write 
  	\begin{equation*}
  		d_\fol(c,c') < (\alpha,\delta)
  	\end{equation*}
  	to mean that there is $t \in [-\alpha,\alpha]$ such that $d(\Phi_t(c),c') < \delta$. 
  \end{definition} 
  
  \begin{example}
    Let $G$ be the fundamental group of a Riemannian manifold $M$.
    Then the sphere bundle $S \tilde M$ equipped with the geodesic flow is a flow space for the fundamental group of $M$. 
    For manifolds of negative or non-positive curvature this flow space is at the heart of the connection between $K$-theory and dynamics used to great effect by Farrell-Jones. 
        
    This example has generalizations to hyperbolic groups and $\CAT(0)$-groups.
    For hyperbolic groups Mineyev's symmetric join is a flow space~\cite{Mineyev-FlowsJoins}.
    Alternatively, it is possible to use a coarse flow space for hyperbolic groups, see Remark~\ref{rem:coarse-flow-space} below.
    For groups acting on a $\CAT(0)$-space a flow space has been constructed in~\cite{Bartels-Lueck-CAT0-geod-flow}. 
    It consists of all parametrized geodesics in the $\CAT(0)$-space (technically all generalized geodesics) and the flow acts by shifting the parametrization.  	
  \end{example}

  Almost equivariant maps often arise as compositions 
  \begin{equation*}
  	X \xrightarrow{\varphi} \FS \xrightarrow{\psi} \Delta^{(N)}_\calf(G),
  \end{equation*}
  where the first map is almost equivariant in an $(\alpha,\delta)$-sense, and the second map is $G$-equivariant and contracts $(\alpha,\delta)$-distances to $\e$-distances.
  The following Lemma summarizes this strategy.
  
  \begin{lemma} \label{lem:N-F-amenable-via-FS}
  	Let $X$ be a $G$-space, where $G$ is a countable group.
  	Let $N \in \IN$.
    Assume that there exists a flow space $\FS$ satisfying the following two conditions.
    \begin{enumerate}
  	  \item[(A)] 
  	     For any finite subset $S$ of $G$ there is $\alpha > 0$ such that for any $\delta > 0$ there is a continuous map $\varphi \colon X \to \FS$ such that for $x \in X$, $g \in S$ we have
  	    \begin{equation*}
  	    	d_\fol(\varphi(gx),g\varphi(x)) < (\alpha,\delta).
  	    \end{equation*}   
  	  \item[(B)] 
  	     For any $\alpha > 0$, $\e > 0$ there are $\delta > 0$ and a continuous $G$-map $\psi \colon \FS \to \Delta^{(N)}_\calf(G)$ such that 
  	     \begin{equation*}
  	     	d_\fol(c,c') < (\alpha,\delta) \quad \implies \quad d(\psi(c),\psi(c')) < \e  
  	     \end{equation*}  
  	     holds for all $c,c' \in \FS$.
  	\end{enumerate}
    Then the action of $G$ on $X$ is $N$-$\calf$-amenable.
  \end{lemma}
  
  \begin{proof}
  	Let $S \subset G$ be finite and $\e > 0$.
  	We need to construct a map $f \colon X \to \Delta_\calf^{(N)}(G)$ for which	$d(f(gx),gf(x)) < \e$ for all $x \in X$, $g \in S$.
  	Let $\alpha$ be as in~(A) with respect to $S$.
    Choose now $\delta > 0$ and a $G$-map $\psi \colon \FS \to \Delta^{(N)}_\calf(G)$ as in (B).
    Next choose $\varphi \colon X \to \FS$ as in (A) with respect to this $\delta > 0$.
    Then $f := \psi \circ \varphi$ has the required property.
  \end{proof}  
    
  \begin{remark}[On the constructions of $\varphi$] \label{rem:(A)-for-negative-curvature}
  	Maps $\varphi \colon X \to \FS$ as in condition~(A) in Lemma~\ref{lem:N-F-amenable-via-FS} can in negatively or non-positively curved situations often be constructed using dynamic properties of the flow.
  	We briefly illustrate this in a case already considered by Farrell and Jones.
  	
  	Let $G$ be the fundamental group of a closed Riemannian manifold of strict negative sectional curvature $M$.
  	Let $\tilde M$ be its universal cover and $S_\infty$ the sphere at infinity for $\tilde M$.
  	The action of $G$ on $\tilde M$ extends to $S_\infty$.
  	For each $x \in \tilde M$ there is a canonical identification between the unit tangent vectors at $x$ and $S_\infty$: every unit tangent vector $v$ at $x$ determines a geodesic ray $c$ starting in $x$, the corresponding point $\xi \in S_\infty$ is $c(\infty)$.
  	One say that $v$ points to $\xi$.
  	The geodesic flow $\Phi_t$ on $S\tilde M$ has the following property.
  	Suppose that $v$ and $v'$ are unit tangent vectors at $x$ and $x'$ pointing to the same point in $S_\infty$.
  	Then $d_\fol(\Phi_t(v),\Phi_{t}(v')) < (\alpha,\delta_t)$ where $\alpha$ depends only on $d(x,x_0)$ and $\delta_t \to 0$ uniformly in $v, v'$ (still depending on $d(x,x_0)$).
  	This statement uses strict negative curvature.
  	(For closed manifolds of non-positive sectional curvature the vector $v'$ has to be chosen more carefully depending on $v$ and $t$; this necessitates the use of the focal transfer~\cite{Farrell-Jones-TopRigigdityCompactNonPos} respectively the use of homotopy coherent actions.)  
  		 	                
  	This contracting property of the geodesic flow can be translated into the construction of maps as in~(A).
  	Fix a point $x_0 \in \tilde M$.
  	Define $\varphi_0 \colon S_\infty \to S\tilde M$ by sending $\xi$ to the unit tangent vector at $x_0$ pointing to $\xi$.  
  	For $t \geq 0$ define $\varphi_t(\xi) := \Phi_t (\varphi_0(\xi))$.
  	Using the contracting property of the geodesic flow is not difficult to check that for any $g \in G$ there is $\alpha > 0$ (roughly $\alpha = d(gx_0,x_0)$) such that for any $\delta > 0$ there is $t_0$ satisfying
  	\begin{equation*}
  		\forall t \geq t_0,\; \forall \xi \in S_\infty \;\; d_\fol(\varphi_t(g \xi),g \varphi_t(\xi)) < (\alpha,\delta).
  	\end{equation*}  
  \end{remark}  
  
  \begin{remark}
  	Of course the space $S_\infty$ used in Remark~\ref{rem:(A)-for-negative-curvature} is not contractible and therefore not a Euclidean retract.
  	But the compactification $\tilde M \cup S_\infty$ of $\tilde M$ is a disk, in particular a Euclidean retract.
  	As $\tilde M$ has the homotopy type of a free $G$-$CW$-complex, there is even a $G$-equivariant map $\tilde M \to \Delta^{(d)}(G)$ where $d$ is the dimension of $M$.
  	In particular, the action of $G$ on $\tilde M$ is $d$-amenable.
  	It is not difficult to combine the two statements to deduce that the action of $G$ on $\tilde M \cup S_\infty$ is finitely $\calf$-amenable.
  	This is best done via the translation to open covers of $G \x (\tilde M \cup S_\infty)$ discussed in Remark~\ref{rem:covers-not-maps}, see for example~\cite{Sawicki-eq-asym-dim}.  
  \end{remark}
    
  An important point in the formulation of condition~(B) is the presence of $\delta > 0$ uniform over $\FS$.
  If the action of $G$ on $\FS$ is cocompact, then a version of the Lebesgue Lemma guarantees the existence of some uniform $\delta > 0$, i.e., it suffices to construct $\psi \colon \FS \to \Delta_\calf^{(N)}(G)$ such that $d(\psi(c),\psi(\Phi_t(c)) < \e$ for all $t \in [-\alpha,\alpha]$, $c \in \FS$.
  
  
  
   
  \begin{remark}[Long thin covers of $\FS$] \label{rem:long-thin-covers}
    Maps $\varphi$ as in condition~(B) of Lemma~\ref{lem:N-F-amenable-via-FS} are best constructed as maps associated to long thin covers of the flow space.
    These long thin covers are an alternative to the long thin cell structures employed by Farrell-Jones~\cite{Farrell-Jones-Ktheory-dynamics-I}. 
  
    An open cover $\calu$ of $\FS$ is said to be an $\alpha$-long cover for $\FS$ if for each $c \in \FS$ there is $U \in \calu$ such that 
    \begin{equation*}
    	\Phi_{[-\alpha,\alpha]}(c) \subseteq U.
    \end{equation*} 	
    It is said to be $\alpha$-long and $\delta$-thick if for each $c \in \FS$ there is $U \in \calu$ containing the $\delta$-neighborhood of $\Phi_{[-\alpha,\alpha]}(c)$.
    The construction of maps $\FS \to \Delta_\calf^{(N)}(G)$ as in condition~(B) of Lemma~\ref{lem:N-F-amenable-via-FS} amounts to finding for given $\alpha$ an $\calf$-cover $\calu$ of $\FS$ of dimension at most $N$ that is $\alpha$-long and $\delta$-thick for some $\delta > 0$ (depending on $\alpha$).
    For cocompact flow spaces such covers can be constructed in relatively great generality~\cite{Bartels-Lueck-Reich-Cover, Kasprowski-Rueping-long-and-thin}.
    Cocompactness is used to guarantee $\delta$-thickness.
    For not cocompact flow spaces on can still find $\alpha$-long covers, but without a uniform thickness,
    they do not provide the maps needed in~(B).       
  \end{remark}

   \begin{remark}[Coarse flow space] \label{rem:coarse-flow-space}
      We outline the construction of the coarse flow space from~\cite{Bartels-coarse-flow} for a hyperbolic group $G$.
      Let $\Gamma$ be a Cayley graph for $G$. 
      The vertex set of $\Gamma$ is $G$.   	  
   	  Adding the Gromov boundary to $G$ we obtain the compact space $\overline{G} = G \cup \dd G$. 
   	  Assume that $\Gamma$ is $\delta$-hyperbolic.
   	  The \emph{coarse flow space} $\CF$ consists of all triples $(\xi_-,v,\xi_+)$ with $\xi_{\pm} \in \overline{G}$ and $v \in G$ such that there is some geodesic from $\xi_-$ to $\xi_+$ in $\Gamma$ that passes $v$ within distance $\leq \delta$.
   	  Informally, $v$ coarsely belongs to a geodesic from $\xi_-$ to $\xi_+$. 
   	  The coarse flow space is the disjoint union of its coarse flow lines $\CF_{\xi_-,\xi_+} := \{ \xi_- \} \x \Gamma \x \{ \xi_+ \} \cap \CF$.
   	  The coarse flow lines are are quasi-isometric to $\IR$ (with uniform constants depending on $\delta$).
   	  
   	  There are versions of the long thin covers from Remark~\ref{rem:long-thin-covers} for $\CF$.  
   	  For $\alpha > 0$ these are $\VCyc$-covers $\calu$ of bounded dimension that are $\alpha$-long in the direction of the coarse flow lines: for $(\xi_-,v,\xi_+) \in \CF$ there is $U \in \calu$ such that $\{ \xi_- \} \x B_\alpha(v) \x \{ \xi_+ \} \cap \CF_{\xi_-,\xi_+} \subseteq U$.
   	  
   	  There is also a coarse version of the map $\varphi_t$ from Remark~\ref{rem:(A)-for-negative-curvature}.
   	  To define it, fix a base point $v_0 \in G$.
   	  For $t \in \IN$, $\varphi_t$ sends $\xi \in \dd G$ to $(v_0,v,\xi)$ where $d(v_0,v) = t$ and $v$ belongs to a geodesic from $v_0$ to $\xi$. 	
   	  It is convenient to extend $\varphi_t$ to a map $G \x \dd G \to \CF$, with $\varphi_t(g,\xi) := (gv_0,v,\xi)$ where now $d(v_0,v) = t$ and $v$ belongs to a geodesic from $gv_0$ to $\xi$.  
   	  Of course $v$ is only coarsely well defined.
   	  Nevertheless, $\varphi_t$ can be used to pull long thin covers for $\CF$ back to $G \x \dd G$.
   	  For $S \subseteq G$ finite there are then $t > 0$ and $\alpha > 0$ such that this yields $S$-wide covers for $G \x \dd G$.
   	  The proof of this last statement uses a compactness argument and it is important at this point that $\Gamma$ is locally finite and that $G$ acts cocompactly on $\Gamma$. 
   \end{remark}

  \section{Covers at infinity} \label{sec:cover-at-infty}
  
  \subsection*{The Farrell-Jones Conjecture for $\GL_n(\IZ)$} 
  
  The group $\GL_n(\IZ)$ is not a $\CAT(0)$-group, but it has a proper isometric action on a $\CAT(0)$-space, the symmetric space $X := GL_n(\IR)/O(n)$.
  Fix a base point $x_0 \in X$.
  For $R \geq 0$ let $B_R$ be the closed ball of radius $R$ around $x_0$.
  This ball is a retract of $X$ (via the radial projection along geodesics to $x_0$) and inherits a homotopy coherent action $\Gamma_R$ from the action of $\GL_n(\IZ)$ on $X$.
  Let $\calf$ be the family of subgroups generated by the virtually cyclic and the \inred{proper} parabolic subgroups of $\GL_n(\IZ)$.
  The key step in the proof of the Farrell-Jones Conjecture for $\GL_n(\IZ)$ in~\cite{Bartels-Lueck-Reich-Rueping-GLnZ} is, in the language of Section~\ref{sec:actions}, the following.
  
  \begin{theorem} \label{thm:GL-fin-amenable}
  	The sequence of homotopy coherent actions $(B_R,\Gamma_R)$ is finitely $\calf$-amenable.
  \end{theorem}
 
  In particular $\GL_n(\IZ)$ satisfies the Farrell-Jones Conjecture relative to $\calf$ by Theorem~\ref{thm:FJ-homotopy-coherent}.
  Using the transitivity principle~\ref{thm:transitivity-principle} the Farrell-Jones Conjecture for $\GL_n(\IZ)$ can then be proven by induction on $n$.
  The induction step uses inheritance properties of the Conjecture and that virtually poly-cyclic groups satisfy the Conjecture. 
  
  The verification of Theorem~\ref{thm:GL-fin-amenable} follows the general strategy of Lemma~\ref{lem:N-F-amenable-via-FS} (in a variant for homotopy coherent actions).  
  The additional difficulty in verifying assumption (B) is that, as the action of $\GL_n(\IZ)$ on the symmetric space is not cocompact, the action on the flow space is not cocompact either.
  The general results reviewed in Section~\ref{sec:flow-spaces} can still be used to construct for any $\alpha > 0$ an $\alpha$-long cover $\calu$ for the flow space.
  However it is not clear that the resulting cover is $\delta$-thick, for a $\delta > 0$ uniformly over $\FS$.   
  The remedy for this short-coming is a second collection of open subsets of $\FS$.
  Its construction starts with an $\calf$-cover for $X$ at $\infty$, meaning here, away from cocompact subsets. 
  Points in the symmetric space can be viewed as inner products on $\IR^n$ and moving towards $\infty$ corresponds to degeneration of inner products along direct summands $W \subset \IZ^n \subset \IR^n$.
  This in turn can be used to define horoballs in the symmetric space, one for each $W$, forming the desired cover~\cite{Grayson-Reduction-semistability}.
  For each $W$ the corresponding horoball is invariant for the parabolic subgroup $\{ g \in GL_n(\IZ) \mid gW = W \}$, more precisely, the horoballs are $\calf$-subsets, but not $\VCyc$-subsets. 
  The precise properties of the cover at $\infty$ are as follows.
    
  \begin{lemma}[\cite{Bartels-Lueck-Reich-Rueping-GLnZ, Grayson-Reduction-semistability}]
  	For any $\alpha > 0$ there exists a collection $\calu_\infty$ of open $\calf$-subsets of $X$ of order $\leq n$ that is of Lebesgue number $\geq \alpha$ at $\infty$, i.e., there is $K \subset X$ compact such that for any $x \in X \setminus \GL_n(\IZ)\cdot K$ there is $U \in \calu_\infty$ containing the $\alpha$-ball $B_\alpha(x)$ in $X$ around $x$. 
  \end{lemma}
  
  This cover can be pulled back to the flow space where it provides a cover at $\infty$ for the flow space that is both (roughly) $\alpha$-long and $\alpha$-thick at $\infty$.
  Then one is left with a cocompact subset of the flow space where the cover $\calu$ constructed first is $\alpha$-long and $\delta$-thick.  
       
  This argument for $GL_n(\IZ)$ has been generalized to $GL_n(F(t))$ for finite fields $F$, and $GL_n(\IZ[S^{-1}])$, for $S$ a finite set of primes~\cite{Rueping-S-arithmetic} using suitable generalizations of the above covers at $\infty$.
  In this case the parabolic subgroups are slightly bigger, in particular the induction step (on $n$) here uses that the Farrell-Jones Conjecture holds for all solvable groups.
  Using inheritance properties and building on these results the Farrell-Jones Conjecture has been verified for all subgroups of $\GL_n(\IQ)$~\cite{Rueping-S-arithmetic} and all lattices in virtually connected Lie groups~\cite{Kammeyer-Lueck-Rueping-FJ-lattices}.

  \subsection*{Relatively hyperbolic groups.} 
  We use Bowditch's characterization of relatively hyperbolic groups~\cite{Bowditch-rel-hyperbolic}.
  A graph is \emph{fine} if there are only finitely many embedded loops of a given length containing a given edge.
  Let $\calp$ be a collection of subgroups of the countable group $G$.
  Then $G$ is hyperbolic relative to $\calp$ if $G$ admits a cocompact action on a fine hyperbolic graph $\Gamma$ such that all edge stabilizers are finite and all vertex stabilizers belong to $\calp$.
  The subgroups from $\calp$ are said to be peripheral or parabolic.
  The requirement that $\Gamma$ is fine encodes Farb's Bounded Coset Penetration property~\cite{Farb-rel-hyp}.
  Bowditch assigned a compact boundary $\Delta$ to $G$ as follows.
  As a set $\Delta$ is the union of the Gromov boundary $\dd \Gamma$ with the set of all vertices of infinite valency in $\Gamma$.
  The topology is the observer topology; a sequence $x_n$ converges in this topology to $x$ if given any finite set $S$ of vertices (not including $x$), for almost all $n$ there is a geodesic from $x_n$ to $x$ that misses $S$.
  (For general hyperbolic graphs this topology is not Hausdorff, but for fine hyperbolic graphs it is.)
  
  The main result from~\cite{Bartels-coarse-flow} is that if $G$ is hyperbolic relative to $\calp$, then $G$ satisfies the Farrell-Jones Conjecture relative to the family of subgroups $\calf$ generated by $\VCyc$ and $\calp$ ($\calp$ needs to be closed under index two supergroups here for this to include the $L$-theoretic version of the Farrell-Jones Conjecture). 
  This result is obtained as an application of Theorem~\ref{thm:FJ-N-F-actions}.
  The key step is the following.
  
  \begin{theorem}[\cite{Bartels-coarse-flow}] \label{thm:Delta-fin-amenable}
  	The action of $G$ on $\Delta$ is finitely $\calf$-amenable.
  \end{theorem}  
  
  This is a direct consequence of Propositions~\ref{prop:cover-theta-small} and~\ref{prop:cover-theta-large} below, using the characterization of $N$-$\calf$-amenability from Remark~\ref{rem:covers-not-maps} by the existence of $S$-wide covers of $G \x \Delta$.
     
  To outline the construction of these covers and to prepare for the mapping class group we introduce some notation.
  Pick a $G$-invariant proper metric on the set $E$ of edges of $\Gamma$; this is possible as $G$ and $E$ are countable and the action of $G$ on $E$ has finite stabilizers.
  For each vertex $v$ of $\Gamma$ with infinite valency let $E_v$ be the set of edges incident to $v$.
  Write $d_v$ for the restriction of the metric to $E_v$.
  For $\xi \in \Delta$, $\xi \neq v$ we define its projection $\pi_v(\xi)$ to $E_v$ as the set of all edges of $\Gamma$ that are appear as initial edges of geodesics from $v$ to $\xi$.
  This is a finite subset of $E_v$ (this depends again of fineness of $\Gamma$).
  Fix a vertex $v_0$ of finite valence as a base point.
  For $g \in G$, $\xi \in \Delta$  define their \emph{projection distance} at $v$ by
  \begin{equation*}
  	d_v^\pi(g,\xi) := d_v ( \pi_v(gv_0), \pi_v(\xi)).
  \end{equation*}
  For $\xi = v$, set $d^\pi_v(g,v) := \infty$.
  (For relative hyperbolic groups a related quantity is often called an \emph{angle}; the terminology here is chosen to align better with the case of the mapping class group.)
  If we vary $g$ (in a finite set) and $\xi$ (in an open neighborhood) then for fixed $v$ the projection distance $d^\pi_{v}(g,\xi)$ varies by a bounded amount.
  Useful is the following attraction property for projection distances: there is $\Theta_0$ such that if $d^\pi_v(g,\xi) \geq \Theta_0$, then any geodesic from $gv_0$ to $\xi$ in $\Gamma$ passes through $v$.
  Conversely, if some geodesic from $gv_0$ to $\xi$ misses $v$, then $d^\pi_v(g,\xi) < \Theta'_0$ for some uniform $\Theta'_0$.
      
  Projection distances are used to control the failure of $\Gamma$ to be locally finite.
  In particular, provided all projection distances are bounded by a constant $\Theta$, a variation of the argument for hyperbolic groups (using a coarse flow space), can be adapted to provide $S$-long covers for the $\Theta$-small part of $G \x \Delta$.
  The following is a precise statement.
  
  \begin{proposition} \label{prop:cover-theta-small}
  	There is $N$ (depending only on $G$ and $\Delta$) such that for any $\Theta > 0$ and any $S \subseteq G$ finite there exists a collection $\calu$ of open $\VCyc$-subsets of $G \x \Delta$ that is $S$-wide on the $\Theta$-small part, i.e., if $(g,\xi) \in G \x \Delta$ satisfies $d^\pi_v(g,\xi) \leq \Theta$ for all vertices $v$, then there is $U \in \calu$ with $gS \x \{ \xi \} \subseteq U$. 
  \end{proposition}
    
  To deal with large projection distances  an explicit construction can be used (similar to the case of $\GL_n(\IZ)$).
  For $(g,\xi) \in G \x \Delta$ let
  \begin{equation*}
  	V_\Theta(g,\xi) := \{ v \mid d_v^\pi(g,\xi)  \geq \Theta \}.
  \end{equation*}
  As a consequence of the attraction property, for sufficiently large $\Theta$, the set $V_\Theta(g,\xi)$ consists of vertices that belong to any geodesic from $gv_0$ to $\xi$.
  In particular, it can be linearly ordered by distance from $gv_0$.
  
  For a fixed vertex $v$ and $\Theta > 0$ define $W(v,\Theta) \subset G \x \Delta$ as the (interior of the) set of all pairs $(g,\xi)$ for which $v$ is minimal in $V_\Theta(g,\xi)$, i.e., $v$ is the vertex closest to $gv_0$ for which $d^\pi_v(g,\xi) \geq \Theta$.
  Then $\calw(\Theta) := \{ W(v,\Theta) \mid v \in V \}$ is a collection of pairwise disjoint open $\calp$-subsets of $G \x \Delta$.  
  
  \begin{proposition} \label{prop:cover-theta-large}
  	Let $S \subset G$ be finite.
  	Then there are $\theta'' \gg \theta' \gg \theta \gg 0$ such that $\calw(\theta) \cup \calw(\theta')$ is a $G$-invariant collection of open $\calp$-subsets of order $\leq 1$ that is $S$-long on  the $\theta''$-large part of $G \x \Delta$: if $d^\pi_v(g,\xi) \geq \theta''$ for some vertex $v$, then there is $W \in \calw(\theta) \cup \calw(\theta')$ such that $gS \x \{ \xi \} \subset W$.
  \end{proposition} 

  A difficulty in working with the  $\calw(v,\Theta)$ is that it is for fixed $\Theta$ not possible to control exactly how $V_\Theta(g,\xi)$ varies with $g$ and $\xi$.
  In particular whether or not a vertex $v$ is minimal in $V_\Theta(g,\xi)$ can change under small variation in $g$ or $\xi$.
  A consequence of the attraction property that is useful for the proof of Proposition~\ref{prop:cover-theta-large} is the following: suppose there are vertices $v_0$ and $v_1$ with $d^\pi_{v_i}(g,\xi) > \Theta' \gg \Theta$, then the segment between $v_0$ and $v_1$ in the linear order of $V_\Theta(g,\xi)$ is unchanged under suitable variations of $(g,\xi)$ depending on $\Theta'$.    
  
  \begin{remark}
  	A motivating example of relatively hyperbolic groups are fundamental groups $G$ of complete Riemannian manifolds $M$ of pinched negative sectional curvature and finite volume.
  	These are hyperbolic relative to their virtually finitely generated nilpotent subgroups~\cite{Bowditch-rel-hyperbolic, Farb-rel-hyp}.
  	In this case we can work with the sphere at $\infty$ of the universal cover $\tilde M$ of $M$.
 	The splitting of $G \x S_\infty$ into a $\Theta$-small part and a $\Theta$-large part can be thought of as follows.
  	Fix a base point $x_0 \in \tilde M$.
  	Instead of a number $\Theta$ we choose a cocompact subset $G\cdot K$ of $\tilde M$.
  	The small part of $G \x \Delta$ consists then of all pairs $(g,\xi)$ for which the geodesic ray from $gx_0$ to $\xi$ is contained in $X$; the large part is the complement.
  	Under this translation the cover from Proposition~\ref{prop:cover-theta-large} can again be thought of as a cover at $\infty$ for $\tilde M$.
  	Moreover, the vertices of infinite valency in $\Gamma$ correspond to horoballs in $\tilde M$, and  projection distances to time geodesic rays spend in horoballs.
  	
    Note that the action of $G$ on the graph $\Gamma$ in the definition of relative hyperbolicity we used is cocompact, but $\Gamma$ is not a proper metric space.
    Conversely, in the above example the action of $G$ on $\tilde M$ is no longer cocompact, but now $\tilde M$ is a proper metric space.
    A similar trade off (cocompact action on non proper space versus non-cocompact action on proper space) is possible for all relatively hyperbolic groups~\cite{Groves-Manning-Dehn-filling}, assuming the parabolic subgroups are finitely generated.
  \end{remark}

  \subsection*{The mapping class group}  
  Let $\Sigma$ be a closed orientable surface of genus $g$ with a finite set $P$ of $p$ marked points.
  We will assume $6g + 2p - 6 > 0$.
  The mapping class group $\Mod(\Sigma)$ of $\Sigma$ is the group of components of the group of orientation preserving homeomorphisms of $\Sigma$ that leave $P$ invariant. 
  Teichm\"uller space $\calt$ is the space of marked complete hyperbolic structures of finite area on $\Sigma \setminus P$. 
  The mapping class group acts on Teichm\"uller space by changing the marking.  
  Thurston~\cite{Fathi-Laudenbach-Poenaru-surfaces-Thurston-english} defined an equivariant compactification of Teichm\"uller space $\overline \calt$.
  As a space $\overline \calt$ is a closed disk, in particular it is an Euclidean retract.
  The boundary of the compactification $\PMF := \overline{\calt} \setminus \calt$ is the space of projective measured foliations on $\Sigma$.
  The key step in the proof of the Farrell-Jones Conjecture for $\Mod(\Sigma)$ is the following.
  
  \begin{theorem}[\cite{Bartels-Bestvina-FJC-MCG}] \label{thm:PMF-amenable}
  	Let $\calf$ be the family of subgroups of $\Mod(\Sigma)$ that virtually fix a point in $\PMF$.
  	The action of $\Mod(\Sigma)$ on $\PMF$ is finitely $\calf$-amenable.
  \end{theorem}
  
  From this it follows quickly that the action on $\overline{\calt}$ is finitely $\calf$-amenable as well, and applying Theorem~\ref{thm:FJ-N-F-actions} we obtain the Farrell-Jones Conjecture for $\Mod(\Sigma)$ relative to $\calf$.
  Up to passing to finite index subgroups, the groups in $\calf$ are central extensions of products of mapping class groups of smaller complexity.
  Using the transitivity principle and inheritance properties one then obtains the Farrell-Jones Conjecture for $\Mod(\Sigma)$ by induction on the complexity of $\Sigma$.
  The only additional input in this case is that the Farrell-Jones Conjecture holds for finitely generated free abelian groups.
     
  The proof of Theorem~\ref{thm:PMF-amenable} uses the characterization of $N$-$\calf$-amenability from Remark~\ref{rem:covers-not-maps} and provides suitable covers for $\Mod(\Sigma) \x \PMF$.
  Similar to the relative hyperbolic case the construction of these covers is done by splitting $\Mod(\Sigma) \x \PMF$ into two parts.
  Here it is natural to refer to these parts as the thick part and the thin part.
  (The thick part corresponds to the $\Theta$-small part in the relative hyperbolic case.) 
  
  Teichm\"uller space has a natural filtration by cocompact subsets. 
  For $\e > 0$ the $\e$-thick part $\calt_{\geq \e} \subseteq \calt$ consist of all marked hyperbolic structures such that all closed geodesics have length $\geq \e$. 
  The action of $\Mod(\Sigma)$ on $\calt_{\geq \e}$ is cocompact~\cite{Mumford-Remark-Mahler-compactness}.
  Fix a base point $x_0 \in \calt$.
  Given a pair $(g,\xi) \in \Mod(\Sigma) \x \PMF$ there is a unique Teichm\"uller ray $c_{g,ξ}$ that starts at $g(x_0)$ and is ``pointing towards $\xi$'' (technically, the vertical foliation of the quadratic differential is $\xi$).
  The $\e$-thick part of $\Mod(\Sigma) \x \PMF$ is defined as the set of all pairs $(g,\xi)$ for which the Teichm\"uller ray $c_{g,\xi}$ stays in $\calt_{\geq \e}$. 
  
  An important tool in covering both the thick and the thin part is the complex of curves $\calc(\Sigma)$.
  A celebrated result of Mazur-Minsky is that $\calc(\Sigma)$ is hyperbolic~\cite{Masur-Minsky-Curve-Complex-Hyperbolicity}.
  Klarreich~\cite{Klarreich} studied a coarse projection map $\pi \colon \calt \to \calc(\Sigma)$ and identified the Gromov boundary $\dd \calc(\Sigma)$ of the curve complex.
  In particular, the projection map has an extension $\pi \colon \PMF \to \calc(\Sigma) \cup \dd \calc(\Sigma)$.
  (On the preimage of the Gromov boundary this extension is a continuous map; on the complement it is still only a coarse map.)
   
  Teichm\"uller space is not hyperbolic, but its thick part $\calt_{\geq \e}$ has a number of hyperbolic properties:
  The Masur criterion~\cite{Masur-Hausdorff-dim-nonergodic-foliations} implies that for $(g,\xi)$ in the thick part, $c_{g,\xi}(t) \to \xi$ as $t \to \infty$.
  Moreover, the restriction of Klarreichs projection map $\pi \colon \PMF \to \calc(\Sigma) \cup \dd \calc(\Sigma)$ to the space of all such $\xi$ is injective. 
  A result of Minsky~\cite{Minsky-Quasi-Proj-Teichmueller} is that geodesics $c$ that stay in $\calt_{\geq \e}$ are contracting. 
  This is a property they share with geodesics in hyperbolic spaces: the nearest point projection $\calt \to c$ maps balls disjoint from $c$ to uniformly bounded subsets.
  Teichm\"uller geodesics in the thick part $\calt_{\geq \e}$ project to quasi-geodesics in the curve complex with constants depending only on $\e$. 
  All these properties eventually allow for the construction of suitable covers of any thick part using a coarse flow space and methods from the hyperbolic case.
  A precise statement is the following.
  
  \begin{proposition}
  	\label{prop:cover-thick}
  	There is $d$ such that for any $\e > 0$ and any $S \subset \Mod(\Sigma)$ finite there exists a $\Mod(\Sigma)$-invariant collection $\calu$ of $\calf$-subsets of $\Mod(\Sigma) \x \PMF$ of order $\leq d$ such for any $(g,\xi)$ for which $c_{g,\xi}$ stays in $\calt_{\geq \e}$ there is $U \in \calu$ with $gS \x \{ \xi \} \subseteq U$. 
  \end{proposition}
  
  The action of the mapping class group on the curve complex $\calc(\Sigma)$ does not exhibit the mapping class group as a relative hyperbolic group in the sense discussed before; the $1$-skeleton of $\calc(\Sigma)$ is not fine.
  Nevertheless, there is an important replacement for the projections to links used in the relatively hyperbolic case, the subsurface projections of Masur-Minsky~\cite{Masur-Minsky-Curve-Complex-Hierarchie}.
  In this case the projections are not to links in the curve complex, but to curve complexes $\calc(Y)$ of subsurfaces $Y$ of $\Sigma$.
  (On the other hand, often links in the curve complex are exactly curve complexes of subsurfaces.)
  The theory is however much more sophisticated than in the relatively hyperbolic case.
  Projections are not always defined; sometimes the projection is to points in the boundary of $\calc(Y)$ and the projection distance is $\infty$.
  Bestvina-Bromberg-Fujiwara~\cite{Bestvina-Bromberg-Fujiwara-actions-on-quasi-trees-MCG} used subsurface projections to prove that the mapping class group has finite asymptotic dimension.
  In their work the subsurfaces of $\Sigma$ are organized in a finite number $N$ of families $\bfY$ such that two subsurfaces in the same family will always intersect in an interesting way.
  This has the effect that the projections for subsurfaces in the same family interact in a controlled way with each other.
  Each family $\bfY$ of subsurfaces is organized in~\cite{Bestvina-Bromberg-Fujiwara-actions-on-quasi-trees-MCG} in an associated simplicial complex, called the projection complex. 
  The vertices of the projection complex are the subsurfaces from $\bfY$.   
  A perturbation of the projection distances can thought of as being measured along geodesics in the projection complex and now behaves very similar as in the relative hyperbolic case, in particular the attraction property is satisfied in each projection complex.
  This allows the application of a variant of the construction from Proposition~\ref{prop:cover-theta-large} for each projection complex that eventually yield the following.
  
  \begin{proposition}
  	\label{prop:cover-large-proj}
  	Let $\bfY$ be any of the finitely many families of subsurfaces.
  	For any $S \subseteq G$ finite there exists $\Theta > 0$ and a $\Mod(\Sigma)$-invariant collection $\calu$ of $\calf$-subsets of $\Mod(\Sigma) \x \PMF$ of order $\leq 1$ such for any $(g,\xi)$ for which there is $Y \in \bfY$ with $d^\pi_Y(g,\xi) \geq \Theta$ there is $U \in \calu$ with $gS \x \{ \xi \} \subseteq U$. 
  \end{proposition}
  
  The final piece, needed to combine Propositions~\ref{prop:cover-thick} and~\ref{prop:cover-large-proj} to a proof of Theorem~\ref{thm:PMF-amenable}, is a consequence of Rafi's analysis of short curves in $\Sigma$ along Teichm\"uller rays~\cite{Rafi-short-curves-Teichmueller-geodesic}: for any $\e > 0$ there is $\Theta$ such that if some curve of $\Sigma$ is $\e$-short on $c_{g,\xi}$ (i.e., if $c_{g,\xi}$ is not contained in $\calt_{\geq \e}$) then there is a subsurface $Y$ such that $d_Y^\pi(g,\xi) \geq \Theta$.

  \begin{remark}
    Farrell-Jones~\cite{Farrell-Jones-Rigidity-GL_m} proved topological rigidity results for fundamental groups of non-positively curved manifolds that are in addition $A$-regular.
    The latter condition bounds the curvature tensor and its covariant derivatives over the manifold.
    All torsion free discrete subgroups of $\GL_n(\IR)$ are fundamental groups of such manifolds.
    Similar to the examples discussed in this section a key difficulty in~\cite{Farrell-Jones-Rigidity-GL_m} is that the action of the fundamental group $G$ of the manifold on the universal cover is not cocompact.
    The general strategy employed by Farrell-Jones seems however different, in particular, it does not involve an induction over some kind of complexity of $G$.
    The only groups that are considered in an intermediate step are polycyclic groups, and the argument directly reduces from $G$ to these and then uses computations of $K$-and $L$-theory for polycyclic groups.
    
    This raises the following question: Can the family of subgroups in Theorem~\ref{thm:GL-fin-amenable} be replaced with the family of virtually polycyclic subgroups? 
    Recall that the cover constructed for the flow space at $\infty$ is both $\alpha$-long and $\alpha$-thick, while only $\delta$-thickness is needed. 
    So it is plausible that there exist thinner covers at $\infty$ that work for the family of virtually polycyclic subgroups.
     
    For the mapping class group the family from Theorem~\ref{thm:PMF-amenable} can not be chosen to be significantly smaller; all isotropy groups for the action have to appear in the family. 
    But one might ask, whether there exist $N$-$\calf$-amenable actions (or $N$-$\calf$-amenable sequences of homotopy coherent actions) of mapping class groups on Euclidean retracts, where $\calf$ is smaller than the family used in Theorem~\ref{thm:PMF-amenable}.        	
  \end{remark}

\ignore{  ======================PAPIERKORB=========================

} 


\end{document}